\documentclass[12, oneside]{amsart}
\usepackage{amsfonts,amssymb,amscd,amsmath,enumerate,verbatim,calc}
\usepackage[all]{xy}
\usepackage{color}

\newcommand{\CM}{Cohen-Macaulay}

\newcommand{\ff}{\text{if and only if}}
\newcommand{\wrt}{with respect to}

\newcommand{\n}{\mathfrak{n} }
\newcommand{\m}{\mathfrak{m} }

\newcommand{\Ass}{\operatorname{Ass}}
\newcommand{\grade}{\operatorname{grade}}
\newcommand{\depth}{\operatorname{depth}}

\newcommand{\htt}{\operatorname{ht}}

\newcommand{\Spec}{\operatorname{Spec}}
\newcommand{\rank}{\operatorname{rank}}

\newcommand{\projdim}{\operatorname{projdim}}
\newcommand{\injdim}{\operatorname{injdim}}
\newcommand{\Hom}{\operatorname{Hom}}
\newcommand{\ext}{\operatorname{Ext}}
\newcommand{\Tor}{\operatorname{Tor}}

\newcommand{\syz}{\operatorname{Syz}}
\newcommand{\curv}{\operatorname{curv}}
\newcommand{\cx}{\operatorname{cx}}

\theoremstyle{plain}

\newtheorem{theorem}{Theorem}[section]
\newtheorem{corollary}[theorem]{Corollary}
\newtheorem{lemma}[theorem]{Lemma}
\newtheorem{proposition}[theorem]{Proposition}

\theoremstyle{definition}

\theoremstyle{remark}

\AtBeginDocument{%
    \def\MR#1{}
}
\begin{document}

\title{Bass and Betti Numbers of $A/I^n$ }

\author{Ganesh~S.~Kadu}
\address{Department of Mathematics, Savitribai Phule Pune University, Pune 411 007, India}
\email{ganeshkadu@gmail.com}
\author{Tony~J.~Puthenpurakal}
\address{Department of Mathematics, IIT Bombay, Powai, Mumbai 400 076, India}
\email{tputhen@math.iitb.ac.in}
\thanks{The author thanks  DST-SERB for financial support.}

\date{\today}

\begin{abstract}
 Let $(A, \m, k)$ be a Gorenstein local ring of dimension $ d\geq 1.$  Let $I$ be an ideal of $A$ with $\htt(I) \geq d-1.$
 We prove that the numerical function
 \[ n \mapsto
\ell(\ext_A^i(k, A/I^{n+1}))\]
 is given by a polynomial of degree $d-1 $ in the  case when
   $ i \geq d+1 $ and $\curv(I^n) > 1$ for all $n \geq 1.$
 We prove a similar result for the numerical function
 \[ n \mapsto \ell(\Tor_i^A(k,
A/I^{n+1}))\] under the assumption that $A$ is a \CM ~ local ring.
 \noindent We note that there are many examples of ideals satisfying the condition $\curv(I^n) > 1,$  for all $  n \geq 1.$
 We also consider more general functions $n \mapsto \ell(\Tor_i^A(M,
 A/I_n)$ for a filtration $\{I_n \}$ of ideals in $A.$ We prove similar results in the case when $M$ is a maximal \CM ~ $A$-module and   $\{I_n=\overline{I^n} \}$ is the integral closure filtration, $I$ an $\m$-primary ideal in $A.$
\end{abstract}
\maketitle
\section{\bf Introduction}
 Let $(A, \m, k)$ be a Noetherian local ring of dimension $d \geq 1.$ It is well known that for a finitely generated $A$-module $M$ the numerical function
   \[ n \mapsto \ell(\ext^i_A(k, M/I^{n+1}M))\]
   is given by a polynomial of degree at most $d-1$  for $ i \geq 1,$ see \cite[Theorem 2]{Kod}. We denote this polynomial by $\varepsilon^i_{M,I}(x).$
   When $M=A$ we simply denote this by
   $\varepsilon^i_{I}(x).$ Note that $\ell(\ext^i_A(k,
   M/I^{n+1}M))$ is the $i^{th}$ Bass number of $M/I^{n+1}M.$ Dually we have the
   numerical function
 \[ n \mapsto \ell(\Tor_i^A(k, M/I^{n+1}M))\]
giving the $i^{th}$ Betti numbers of $M/I^{n+1}M.$ Again by a
theorem of Kodiyalam \cite[Theorem 2]{Kod}, this is given by a
polynomial of degree at most $d-1.$ We denote this polynomial by
$t_{I, i}^A(M,x).$ These polynomials are collectively called the
Hilbert polynomials associated to derived functors because they
generalize the usual Hilbert polynomial. It is of some
interest to find the degrees of the polynomials
$\varepsilon^i_{M,I}(x)$ and $t_{I, i}^A(M,x).$ For instance in
\cite[Theorem 18]{Pu1} it was proved that if $M$ is a maximal \CM \
$A$-module and $I = \m$ then
 \[
 \deg t_{\m , 1}^A(M,x) < d -1 \quad \text{\ff} \quad M \ \text{is free}.
 \]
 In \cite[Theorem I]{Pu-Srikanth} this result was generalized to
arbitrary finitely generated modules with projective dimension at
least $1$ over \CM ~ local rings. In \cite{Ganesh} it was proved
that for an ideal $I$ of analytic deviation one \[ \deg t_{I,
1}^A(M, x) < d-1 ~ \text{then} ~ F_I(M) ~ \text{is free} ~
F(I)\text{-module}\] where $F(I)= \bigoplus_{n\geq 0}I^n/\m I^n$
is the fiber cone of $I$ and
 $F_I(M) = \bigoplus_{n\geq 0}I^nM/\m I^nM$ is the fiber module of $M$ \wrt \ $I.$
 Katz and Theodorescu in \cite{KTheo} prove that if $I$ integrally closed  then the degree of $t_{I, i}^A(k,x)$ is equal to analytic spread of $I$
  minus one under some mild conditions on ring $A.$
 In the case of Hilbert polynomials associated to extension
 functor see \cite{KTheo3}, \cite{KTheo2}, \cite{Theo}   giving estimates on the degree of the polynomials in some cases of interest.

 In this paper we provide  a new class of ideals, namely ideals $I$ satisfying the condition $\curv(I^n) >1$ for all $n \geq 1$ ( see 2.2 for definition of curvature) for which the the
 numerical functions giving the $i^{th}$ Bass numbers and $i^{th}$ Betti numbers of
 $A/I^n$ are polynomials of degree equal to $d-1$ for $n >>0.$ It follows from \cite[Corollary 5]{LA2} that there many examples
 of ideals with $\curv(I^n) >1$ for all $n \geq 1,$ we note this in 2.4 and 2.5 of section on preliminaries. Our results in the case
 of Bass numbers require the Gorenstein hypothesis on ring $A$
 while in the case of Betti numbers we require $A$ to be \CM. More
 precisely our results state:

   \begin{theorem}\label{mainth1}
    Let $(A, \m, k)$ be a Gorenstein local ring of dimension $ d\geq 1.$  Let $I$ be an ideal of $A$ with $\htt(I) \geq d-1$ and  $\curv(I^n) > 1$ for
     all $n \geq 1.$  Let $x_1, \ldots, x_{d-1}$ be a superficial sequence w.r.t. $I$ and $A_l(\underline x)= A/(x_1, \ldots, x_{d-l})$ for $1 \leq l \leq d.$
      Then for $ i \geq d+1,$ we have $\deg \varepsilon^i_{A_l, I}(x) =l-1.$ In particular when $l=d,$ we get $\deg \varepsilon^i_{I}(x)=d-1.$
     \end{theorem}

\begin{theorem} \label{mainth2}  Let $(A, \m)$ be a \CM ~  ring of dimension $d \geq1.$ Let $I$ be an ideal with $\htt(I) \geq d-1$ and $\curv(I^n) >1 $ for all $ n \geq 1. $
 Then for $i \geq d-1,$
    \[ n \mapsto   \ell ( \Tor_i^A(k, A/I^n)) \]
    is given by a polynomial  $t^A_{I,i}(k, z)$ of degree $d-1$ for $ n >> 0.$
\end{theorem}
We then consider more general functions for $i \geq 1$ 
\vspace{-3.78 mm}
\[ n \mapsto  \ell(\Tor_i^A(M, A/I_n)) \]
where $\mathcal J = \{I_n\}$ is an $I$-admissible filtration of ideals in $A.$ Here we assume that $A$ is \CM ~ of dimension $ d  \geq 1$ and that $M$ is
a non-free maximal \CM ~ $A$-module. Our main results show that
these functions are given by polynomials of degree equal to $d-1$
for the following cases of filtrations of ideals $ \{ I_n\}$:
\begin{enumerate}
  \item [\rm (i)]  $\mathcal J = \{I_n =I^n\}$ is an $I$-adic filtration where $I$ is $\m$-primary
  with $r(I)=1$ i.e. ideals having reduction number one and $ \Tor_i^A(M,
  A/I )\neq 0.$
  \item [\rm (ii)] $\mathcal J = \{I_n =\overline {I^n}\}$ is the
  integral closure filtration of ideals in $A$ where $A$ is analytically unramified
  ring.
 \end{enumerate}
  Moreover we also identify the normalized leading coefficient in
 the case (i) above. In case (ii) we first prove the result in
 dimension one by using the fact that for integral closed $\m$-primary
 ideal quotient $I,$ ring
 $A/I$ acts as a test module for finite projective dimension, see
 \cite[Corollary 3.3]{HunCorso}. We then prove the result by
 induction on dimension $d$ by using Lemma \ref{Pu}.\\

  Here is an overview of the contents of the paper.
   In section 2  on preliminaries we give all the basic definitions,
  notations and also discuss some preliminary facts that we need.
  In section 3 we estimate the degree of the polynomial $\varepsilon^i_{I}(x)$ in the case when  $A$ is Gorenstein local ring of dimension one.
   In section 4 we proceed to estimate the degree of
  $\varepsilon^i_{I}(x)$ in the general case of Gorenstein local ring of arbitrary
  dimension. This proves one of our main theorems \ref{mainth1}  stated above. In section 5 we prove theorem \ref{mainth2} showing that
   the degree of the polynomial $t^A_{I,i}(k,
 z)$ giving the Betti numbers of $A/I^n$ is $d-1$ in the case of
 interest. In section 6 we consider more general Hilbert
 polynomials associated to derived functors of torsion functor. In this section we provide some conditions on ideal $I$ and module $M$ under which
  the associated Hilbert polynomials have
 degree exactly $d-1.$ Finally in section 7 we consider the case
 of integral closure filtration when the ring $A$ is analytically
 unramified \CM.  We prove in theorem \ref{mainth}  that, in this case again, the degree of the associated Hilbert polynomial attains the upper bound of  $d-1.$

\section{\bf Preliminaries}
Throughout  this section we assume that that $(A, \m)$ is a Noetherian local ring  with  residue field $k=A/\m $ and $M$ a finitely generated $A$-module.
 The $n^{th}$ Betti number of $A$-module $M$ is denoted by  $\beta_n^A(M)$ while the $n^{th}$ Bass number of $M$ is denoted by $\mu_A^n(M).$ We first define
 the notions of complexity and  curvature of modules. We also mention some of their basic properties needed in this paper. For detailed proofs and other
 additional information, see \cite[Section 4]{LA}.\\

\noindent {\bf 2.1.} The  complexity of a finitely generated $A$-module $M$ is defined by

\[ \cx M = \inf \bigg\{ d \in \mathbb N \Bigm|
\begin{array}{l}
\text{ there exists polynomial} ~ p(n) ~\text{of degree } d-1   \\
\text{such that} ~ \beta_n^A(M) \leq p(n) \text{ for all} ~ n >> 0.
\end{array}     \bigg\}
\]

\noindent {\bf 2.2.} The curvature of $A$-module $M$ is defined as
\[ \curv M = \limsup_{n \rightarrow \infty} \sqrt[n]{\beta_n^A(M)} \]

\noindent {\bf 2.3.} We state some of the basic properties of complexity and curvature  for finitely generated modules, see \cite[Section 4]{LA} for
detailed proofs.
\begin{itemize}
\item[(1)] $\projdim_A M < \infty \iff  \cx_A M =0 \iff \curv_A M =0. $
\item[(2)] $\projdim_A M = \infty \iff \cx_A M \geq 1 \iff \curv_A M \geq 1 .$
\item[(3)] $ \cx_A M \leq 1 \iff M \text{ has bounded Betti numbers }.$
\item[(4)] $ \cx_A M  < \infty  \implies  \curv_A M \leq 1.$
\item[(5)] $\curv_A M < \infty.$
\item[(6)] $\cx_A M \leq \cx_A k$ and $\curv_A M \leq \curv_A k.$
\item[(7)]  For all $ n\geq 1, ~ \cx_A M = \cx_A \syz_n^A(M) $ and $\curv_A M = \curv_A \syz_n^A(M). $
\item[(8)] Let $\underline x$ be an $A$-regular sequence of length $r$ in $A.$ Set  $A' = A/(\underline x)$ and $M'= M/(\underline x) M.$
Then we have  $\cx_A M \leq \cx_{A'} M' \leq \cx_A M +r; $
  and if in addition $ \projdim_A M= \infty, $ then ~$\curv_{A'}M' = \curv_A M.$
\end{itemize}

\noindent {\bf 2.4.} Let $(A, \m, k)$ be a Noetherian local ring
of dimension $ d\geq 1$ that is not a complete intersection. Let
$J$ be any ideal with $\htt(J) > 0 $ and  for $i \geq 1$ let $I =
\m^iJ.$ Since $A$ is not complete intersection so by \cite[Theorem
3]{LA2},
 $\curv k > 1.$ Now by  \cite[Corollary 5]{LA2}, we obtain $\curv (I^n) > 1$ for all $n \geq 1.$
 Thus there are many examples of ideals satisfying the condition $\curv(I^n) > 1$ for all $n \geq
 1.$\\

\noindent {\bf 2.5.} Let $A$ be a Noetherian local ring. Let $P$
be a prime ideal in $A$ such that $A_P$ is not a complete
intersection and $\dim A_P \geq 1.$ By  2.4 above we have
$\curv_{A_P}(P^nA_P) > 1$ for all $n.$ As $ \curv(P^n) \geq
\curv_{A_P}(P^nA_P) $ we have $\curv(P^n) >1$ for all $n.$\\

\noindent {\bf 2.6.}    We denote by $\mathcal R (I) =
\displaystyle \bigoplus_ {n \geq 0}I^n$ the Rees algebra of ideal
$I.$  For a  finitely generated $A$-module $M$ we denote by
$\mathcal R (I, M) = \displaystyle \bigoplus_ {n \geq 0}I^nM$ the
Rees module of $I$ with respect to $M.$
 Note that $\mathcal R(I, M)$ is finitely generated graded $\mathcal R(I)$-module. The graded $k$-algebra $\mathcal F(I)= \mathcal R (I) \otimes A/\m$ is
 known as the fiber cone of the ideal $I.$ The Krull dimension of $\mathcal F(I)$ is known as the analytic spread of $I,$ denoted by $l(I).$ \\

\noindent {\bf 2.7.} Set $L^I(M) = \displaystyle \bigoplus_{n \geq
0} M / I^{n+1}M.$ Consider the following short exact sequence
\[ 0 \rightarrow \mathcal R(I, M) \rightarrow M[t] \rightarrow L^I(M)(-1)\rightarrow 0.\]
This gives  $L^I(M)(-1) $ and consequently $L^I(M), $ the structure of a graded $\mathcal R(I)$-module. We note that $L^I(M)$ is not a finitely generated
$\mathcal R (I)$-module. This module was introduced in \cite{Tony-Ratliff}.\\
\noindent {\bf 2.8.} Let $I$ be an ideal in $A$ and $M$ be an
$A$-module. An element $x \in I$ is said to be $I$-superficial for
$M$ if there exists a positive integer $c$  with \[ (I^{n+1}M :_M
x) \cap I^cM = I^nM ~~ \text{for all} ~~ n\geq c.\] A sequence of
elements $\underline x = x_1, x_2, ...,x_n$ is said to be
$I$-superficial if $x_1 $ is $I$-superficial for $A$ and for $ i
\geq 2,$ $x_i$ is
$I$-superficial for $ A/(x_1, ...,x_{i-1}).$\\

\noindent {\bf 2.9.} We now recall the notion of filter regular
sequence and mention a condition which guarantees their existence.
Let $  R = \displaystyle \bigoplus_{n \geq 0} R_n$ be a standard
graded ring and $ M = \displaystyle \bigoplus_{n \geq 0} M_n$ be a
graded $R$-module. An element $ a \in R_1$ is called a filter
regular element if $ \mu_a : M_n \rightarrow M_{n+1}$ given by
$\mu_a(x) = ax$ is injective for $n >> 0.$
 A sequence of elements $ \underline a = a_1, a_2, \ldots,  a_s$ in $R_1$   is called a filter regular sequence w.r.t. $M$ if $a_1$ is filter regular and
 for $ i \geq 2,$  $ a_i$ is filter regular w.r.t. $M/ (a_1, \ldots a_{i-1})M.$ Finally we mention that filter regular sequence w.r.t. $M$ exists  if the residue field $R_0/\m_0$  of $R_0$ is infinite, see  \cite[18.3.10]{BS} for a proof of this result and  for  the other general facts about filter regular sequences.\\
\noindent {\bf 2.10.} Let $ (A, \m)$ be a local ring and $I$ be an
$\m-$primary ideal. Then a sequence of ideals $ \mathcal I = \{
I_n\}_{ n \in \mathbb Z} $ is called an $I$-admissible filtration
if for all $m ,n \in \mathbb Z$ we have
\begin{itemize}
\item[(1)] $I_{n+1} \subseteq I_n$
\item[(2)] $I_n  I_m \subseteq I_{n+m}$
\item[(3)] $I^n \subseteq I_n$
\item[(4)] there exists $ k \in \mathbb N$ such that $I_n \subseteq I^{n-k}$ for $ n \geq k.$

\end{itemize}
\noindent {\bf 2.11.} Let $\overline I$ denote integral closure of
ideal $I.$ If $A$ is analytically unramified then the filtration
$\mathcal I = \{ \overline{ I^n}\}$ is an $I$-admissible filtration of ideals by a theorem of D. Rees \cite{Rees}.\\
\noindent {\bf 2.12.} Let $x \in I $ and Set $B=A/xA $, $J=I/(x)$,
$J_n = I_nB$ and $ \overline {\mathcal I} = \{ J_n \}.$
If $ \mathcal I = \{I_n\}$ is $I$-admissible filtration then $\overline {\mathcal I}$ is $J$-admissible filtration of ideals in $B.$\\
\noindent {\bf 2.13.} The notion of superficial element can be
generalized to modules and for any arbitrary filtration of ideals.
Let $I$ be an ideal in $A$ and let $M$ be an $A$-module. Let
$\mathcal I= \{I_n\}$ be an $I$-admissible filtration of ideals in
$A.$ We say that $ x\in I_1 \setminus I_2$ is $\mathcal
I$-superficial for $M$ if there exists $ c \in \mathbb N$ such
that $ (I_{n+1}M :_M x ) \cap I_cM = I_nM$ for all $n \geq c.$ We
note the following,
\begin{itemize}
\item[(1)] If $\depth M >0$ then it is easy to see that every $\mathcal I-$superficial element for $M$ is also $M$-regular. See \cite[Section 2]{Huck-Mar}  for the case M = A.
\item[(2)] If $x $ is $\mathcal I$-superficial with respect to $M$ and $M$-regular, then by using the Artin-Rees lemma for $M$ and $xM$ one gets $(I_nM :_M x) = I_{n-1}M$ for all $ n >> 0.$ See \cite[p. 8]{Sally} for the case M = A.
\item[(3)] Let $x \in I_1$ be $M$-regular element. Then $x$ is a superficial
element of $I_1$ with respect to $M$ if and only if $I_{n+1}M: x= I_nM$ for $ n >> 0.$ See \cite[8.5.3]{Hun-Swan}.
\end{itemize}

\noindent {\bf Convention:} The degree of the zero polynomial is defined to be $- \infty .$

\section{\bf One Dimensional Case}
We now study the growth of $\varepsilon^i_{I}(x)$ in the case when $(A, \m)$ is a Gorenstein local ring of $d=1$ and $I$ is an ideal of analytic spread $l(I) = 1.$ We first need the following lemma.
\begin{lemma}\label{principal} Let $(A, \m)$ be \CM ~ local ring of dimension $1.$
    Let $I$ be an $\m$-primary
    ideal. If $ \mu(I^n) =1$ for some $n \geq 1$ then $I$ is a
    principal ideal.
\end{lemma}
\begin{proof} We may assume that the residue field of $A$ is infinite. Let $J=(x)$ be a minimal reduction generated by an $I$-superficial element. By graded Noether normalization we know that $F(J)$ is a standard homogeneous $k-$subalgebra of $F(I).$ Note that $F(J)_n= <\bar x^n> \subseteq F(I)_n. $ As $ \mu(I^n) =\dim F(I)_n=1$ and $\bar x^n \neq 0$ so we obtain $F(I)_n= < \bar x^n>.$  So $I^n = (x^n) + \m I^n$ and hence by Nakayama lemma $I^n
    = (x^n).$ Now we claim that the following $k$-linear map is
    injective,
    \[ \frac{I^{n-1}}{\m I^{n-1}} \xrightarrow[\hspace{13pt}]{\mu_x}  \frac{I^n}{\m I^n}.\]
    To do this let $ \bar a \in \frac{I^{n-1}}{\m I^{n-1}}$ and
    suppose $\overline{xa} =0 $ i.e. $xa \in \m I^n = \m (x^n).$ Since
    $x$ is $A$-regular so $a \in \m (x^{n-1}).$ Hence $ a \in \m
    I^{n-1}$ i.e. $\bar a =0.$ Thus $\mu_x$ is injective. This shows
    that $ \mu(I^{n-1}) \leq \mu(I^n) =1.$ Iterating this we get
    $\mu(I) =1.$ So $I$ is a principal ideal.
\end{proof}

\begin{lemma} \label{dimone} Let $( A, \m, k)$ be a Gorenstein local ring of
    dimension $d = 1.$ Let $I $ be an ideal with analytic spread $\l(I) =1,$ then
    \begin{enumerate}[\rm (i)]
        \item if $\htt(I)=1$ and $I$ is non-principal ideal then $\deg  \varepsilon^i_{I}(x) = 0 $ for  all $ i \geq 1.$
        \item if $\htt(I)=0$ then $\deg  \varepsilon^i_{I}(x) = 0 $ for  all $ i \geq 1.$
    \end{enumerate}

\end{lemma}
\begin{proof} As $d=1$ from \cite[corollary 4]{Theo} we know that $\varepsilon^i_{I}(n) = c $ for $ n >>  0.$
    So if $c=0$ then $\ext^i_A(k, A/I^n) = 0$ for all $ n >> 0.$ There are two cases, $ \htt(I) =1$ and $\htt(I) =0.$
    (i)  If $\htt(I)=1$ then $\dim A/I^n =0.$ So if $\varepsilon^i_{I}(n) = 0 $ for some $ i \geq 1,$ then we have by  \cite[3.5.12]{BH}
    that $  \injdim_A A/I^n < \infty .$  Now since
    $A$ is Gorenstein this gives us $\projdim_A A/I^n < \infty.$ As
    $ d=1$  for $ n >> 0$ we find that $ \projdim A/I^n $ is either $0$ or $1$  by Auslander-Buchsbaum formula. If  $ \projdim A/I^n =0 $  then $I^n=0.$
    This contradicts $\htt(I)=1.$ So  $ \projdim A/I^n =1. $ Hence,
    \[ 0 \longrightarrow A^r \longrightarrow A \longrightarrow A/I^n \longrightarrow 0\]
    Since $\rank_A(A/I^n) =0$ we obtain that $r=1.$ So $I^n$ is
    principal. So by Lemma \ref{principal} above $I$ is a principal ideal. This is a contradiction to the fact that $I$ is non-principal ideal.
    Hence $c \neq 0$ and
    $\deg  \varepsilon^i_{I}(x) = 0 $ in this case.\\
    (ii) Suppose $\htt(I)=0.$  In this case $\dim A/I^n =1.$ So if $\varepsilon^i_{I}(n) = 0$ for some $i \geq 2,$  we have by
    \cite[3.5.12]{BH} that $  \injdim_A A/I^n < \infty .$  Repeating the arguments as above we  find that $ \projdim A/I^n $ is either $0$ or $1.$
    If  $ \projdim A/I^n =0 $  then $I^n=0.$ This contradicts the fact that $l(I)=1.$  If $ \projdim A/I^n =1 $ then $ \projdim I^n =0 $ and so  $I^n$ is
    a free $A$-module of rank $1.$ So $I^n=(u)$ for some $u \in A.$ It is then easy to show that $u$ is $A$-regular. So $\grade(I^n) \geq 1. $ Since
    $\htt(I)= \htt(I^n)=\grade I^n $ we have $\htt(I) \geq 1 , $ contradicting $\htt(I)=0.$ Hence we obtain $\ext^i_A(k, A/I^n) \neq 0$ for all $ n >> 0.$  This
    shows that $\deg  \varepsilon^i_{I}(x) = 0 $ for  $ i \geq 2.$  \\
    Now suppose $i=1.$ We claim that $\ext^1_A(k, A/I^n) \neq 0$ for $ n >> 0.$ Suppose $\ext^1_A(k, A/I^n) = 0$ for $n >> 0.$ Then by
    \cite[3.1.13]{BH}, we have $\Hom_A(A/\mathfrak p, A/I^n)=0$ for all $ \mathfrak p \neq \m$ and $n >> 0.$ But then $ \mathfrak p$
    contains $A/I^n$-regular element. So all prime ideals $\mathfrak p \in \Spec A$ contain $A/I^n$-regular element and hence $\Ass(A/I^n) = \emptyset.$
    So $A/I^n=0$  contradicting the fact that $I^n$ is a proper ideal.

    \end{proof}

\section{\bf General Case}
We now do the general case where $(A, \m)$ is a Gorenstein local
ring of dimension $d \geq 1.$ Recall from  ${ 2.7}$  that $L^I(M)$
is a graded $\mathcal R(I)$-module.

   \begin{lemma}\label{fingen} Let $(A, \m)$ be a Gorenstein local ring of dimension $d.$ Let $M$ be a finitely generated $A$-module with $\projdim M < \infty.$ For $i \geq d+1$
    we have \[  \ext^i_A(k, L^I(M)(-1)) \cong \ext^{i+1}_A(k, \mathcal R(I,M))\]
   \end{lemma}
   \begin{proof} Since $A$ is Gorenstein, $\projdim M < \infty $ gives $\injdim M < \infty. $ Consider the following short exact sequence of graded $\mathcal R(I)$-modules

   \[ 0 \rightarrow \mathcal R(I, M) \rightarrow M[t] \rightarrow L^I(M)(-1)\rightarrow 0\]
   Applying $\ext^i_A(k, -)$ we get

   \[  \rightarrow \ext^i_A(k, M[t]) \rightarrow \ext^i_A(k,L^I(M)(-1)) \rightarrow \ext^{i+1}_A(k, \mathcal R(I, M)) \rightarrow \ext^{i+1}_A(k, M[t]) \rightarrow  \]
   Since $ \injdim M < \infty $  for $ i \geq d+1 $ we obtain $\ext^i_A(k, M[t]) \cong \displaystyle \displaystyle \bigoplus_{n \geq 0} \ext^i_A(k, M) =0.$
   Hence
   \[  \ext^i_A(k, L^I(M)(-1)) \cong \ext^{i+1}_A(k, \mathcal R(I,M)).\]

   \end{proof}

Let $(A, \m, k)$ be a Gorenstein local ring of dimension $ d\geq 1.$  Let $I$ be an  ideal of $A$ with $\htt(I) \geq d-1.$
Let $ \underline x = x_1, x_2,...,x_{d-1}$ be a $A-$superficial sequence w.r.t. $I.$
For $1 \leq l \leq d$ we set $$A_l(\underline x) = A/(x_1,...,x_{d-l}).$$

Note that $A_1(\underline x) = A/((x_1,...,x_{d-1}) $ and $A_d(\underline x) = A.$
   \begin{lemma}\label{extnonzero} Let $(A, \m)$ be a Gorenstein local ring of dimension $d\geq 2.$ Let $A_1 = A_1(\underline x).$ If $I$
   is an ideal of $A$ with $\htt(I) \geq d-1$ and $\curv_A(I^n) > 1 $ for some $ n \geq 1$ then for all $ i \geq 1$ we have
   \[\ext^i_A(k, \frac{A_1}{I^nA_1}) \neq 0.\]
   \end{lemma}
   \begin{proof} Let $n \geq 1 $ be such that $\curv(I^n) > 1.$  We  claim   $\projdim_A A_1/I^nA_1 =
   \infty.$ So suppose on the contrary that  $\projdim_A A_1/I^nA_1  < \infty.$
   We first observe that $\frac{A_1}{I^nA_1} \otimes_A A_1 \cong \frac{A_1}{I^nA_1}.$ So by \cite[4.2.5(4)]{LA} we obtain
    \[\cx_{A_1}\frac{A_1}{I^nA_1} \leq \cx_A \frac{A_1}{I^nA_1} + d-1.\]
   Since     $ \projdim_A \frac{A_1}{I^nA_1} < \infty  $ we have $\cx_A \frac{A_1}{I^nA_1} =0 $ and so $ \cx_{A_1}\frac{A_1}{I^nA_1} \leq d-1.$
      By \cite[4.2.3(4)]{LA},    $ \curv_{A_1}\frac{A_1}{I^nA_1} \leq 1.$ Since $\curv(I^n) > 1 ,$ we have $\projdim_A A/I^n =\infty.$ Hence by
       \cite[4.2.5(4)]{LA},
      $\curv_A A/I^n = \curv_{A_1} A_1/I^nA_1.$ Already $ \curv_{A_1}\frac{A_1}{I^nA_1} \leq 1,$ so we have $\curv_A A/I^n \leq 1.$ Also by
       \cite[4.2.4(2)]{LA}, $\curv(A/I^n) = \curv(I^n)$ and so $\curv(I^n) \leq 1.$ This contradicts the fact that $\curv(I^n) > 1.$
       Hence we have $\projdim_A \frac{A_1}{I^nA_1} = \infty.$ In this case, $\injdim_A \frac{A_1}{I^nA_1} = \infty$ as $A$ is Gorenstein.
        Hence from \cite[3.5.12]{BH} we get $\mu_i(\m, \frac{A_1}{I^nA_1}) \neq 0$ for all $i > \dim \frac{A_1}{I^nA_1}.$ \\
        \indent In the first case when $\htt(I) =d$ we have
        $\dim \frac{A_1}{I^nA_1}=0.$
        So in this case we obtain $\ext^i_A(k, \frac{A_1}{I^nA_1}) \neq 0$ for all $i \geq
        1.$ \\
        \indent When $\htt(I) =d-1,$ we have $\dim
        \frac{A_1}{I^nA_1}=1$ and so $\ext^i_A(k, \frac{A_1}{I^nA_1}) \neq
        0$ for all $ i \geq 2. $ Suppose now that $\ext^1_A(k, \frac{A_1}{I^nA_1})
        =0.$ In this case by \cite[3.1.13]{BH}, we have $\Hom_A(A/\mathfrak p, A/I^n)=0$ for all prime ideals $\mathfrak p $ with $ \htt(\mathfrak
        p)=d-1.$ But then $ \mathfrak p$ contains $A/I^n$-regular
        element. In particular all minimal primes of $I^n$ contain
        $A/I^n$-regular element. This is a contradiction. So $\ext^1_A(k,
        \frac{A_1}{I^nA_1}) \neq 0.$ Thus in both the cases we have $ \ext^i_A(k, \frac{A_1}{I^nA_1}) \neq
        0$ for all $i \geq 1.$

   \end{proof}



   \begin{lemma}\label{degzero} Let $(A, \m, k)$ be a Gorenstein local ring of dimension $ d\geq 2.$  Let $I$ be an ideal of $A$ with $\htt(I) \geq d-1$ and
    $\curv(I^n) > 1$ for all $n \geq 1.$ For any $ \underline x = x_1, x_2,...,x_d,$  $A$-superficial sequence
   w.r.t. $I$ and for $ i \geq 1,$ we have  $\deg \varepsilon^i_{ A_1(\underline x), I}= 0.$
   \end{lemma}
 \begin{proof} From \cite[corollary 4]{Theo} we
    know that the numerical function \[ n \longmapsto \ell \Big(\ext^i_A\Big(k, \frac{A_1}{I^nA_1}\Big)\Big)\] is given by a polynomial
    $\varepsilon^i_{A_1,I}(x) $ of degree at most $ \dim A_1 -1 = 0.$ Now since $\curv(I^n) >  1$ for all $n \geq 1,$  by lemma \ref{extnonzero}
    we have $\ext^i_A\Big(k, \frac{A_1}{I^nA_1}\Big) \neq 0$  for all $ n \geq 1.$  Thus  $\varepsilon^i_{A_1,I} $ is a non-zero constant polynomial
    and so $\deg \varepsilon^i_{A_1,I}(x) = 0.$
       \end{proof}

   \begin{theorem}\label{theorem1}
    Let $(A, \m, k)$ be a Gorenstein local ring of dimension $ d\geq 1.$  Let $I$ be an ideal of $A$ with $\htt(I) \geq d-1$ and  $\curv(I^n) > 1$ for
     all $n \geq 1.$  Then for $ i \geq d+1,$ we have $\deg \varepsilon^i_{A_l, I}(x) =l-1.$ In particular when $l=d,$ we get $\deg \varepsilon^i_{I}(x)=d-1.$
     \end{theorem}
   \begin{proof} The case $d=1$  follows from Lemma \ref{dimone}. So now $d \geq 2.$ Note first that we may assume the residue field of $A$ is infinite. Let $A_l= A_l(\underline x).$ We prove the result by induction on $l.$ The result is true for $l=1$ by the Lemma \ref{degzero} above.
   So now assuming the result true for $l \geq 1, $ we prove it for $ l+1.$ Note $A_{l+1} = \frac{A}{(x_1, \ldots , x_{d-l-1})}. $
   By Lemma \ref{fingen} for $ i \geq d+1,$ $\ext^i_A(k, L^I(A_l))$ is finitely generated graded $\mathcal R(I)$-module. So we can choose $ y \in I$ satisfying the following two properties,
 \begin{enumerate}
  \item [\rm (i)] $yt \in \mathcal R(I)_1$ is $\ext^s_A(k, L^I(A_l))$-filter regular for $ s= i, i+1.$
  \item [\rm (ii)] $y$ is $I$-superficial w.r.t. $A_{l+1}.$
   \end{enumerate}

  \noindent We now extend the superficial sequence $ \underline y' = x_1, \ldots , x_{d-l-1},y$ to a maximal $I$-superficial sequence $\underline y.$ Notice that $A_{l+1}(\underline x)= A_{l+1}(\underline y).$ Since $ y $ is $I$-superficial sequence for $A_{l+1}$ we have $I^{n+1}A_{l+1}: y = I^nA_{l+1}$ for $ n >>0.$  This shows that the map $\mu_y : \frac{A_{l+1}}{I^nA_{l+1}}  \longrightarrow \frac{A_{l+1}}{I^{n+1}A_{l+1}}  $ is injective.  Hence the following is a short exact sequence for $ n >> 0,$
      \[ 0 \longrightarrow \frac{A_{l+1}}{I^nA_{l+1}}  \xrightarrow[\hspace{17pt}]{\mu_y} \frac{A_{l+1}}{I^{n+1}A_{l+1}} \longrightarrow \frac{A_l(\underline y)}{I^{n+1}A_l(\underline y)}\longrightarrow 0.\]
      Applying  $\Hom(k, -)$ we get

      \[  \rightarrow \ext^i_A \Big(k, \frac{A_{l+1}}{I^nA_{l+1}} \Big)  \xrightarrow[\hspace{17pt}]{\mu_y}   \ext^i_A \Big( k, \frac{A_{l+1}}{I^{n+1}A_{l+1}}  \Big)  \rightarrow
      \ext^i_A \Big( k, \frac{A_l(\underline y)}{I^{n+1}A_l(\underline y)} \Big) \]

       $\quad  \rightarrow  \ext^{i+1}_A\Big(k,   \frac{A_{l+1}}{I^nA_{l+1}} \Big) \xrightarrow[\hspace{17pt}]{\mu_y} \ext^{i+1}_A \Big( k, \frac{A_{l+1}}{I^{n+1}A_{l+1}}  \Big) \rightarrow  $\\

Since  $yt \in \mathcal R(I)_1$ is $\ext^s_A(k, L^I(A_l))$-filter regular for $ s= i, i+1$ we have the following short exact sequence,
\[  0 \rightarrow \ext^i_A \Big(k, \frac{A_{l+1}}{I^nA_{l+1}} \Big)  \xrightarrow[\hspace{17pt}]{\mu_y}   \ext^i_A \Big( k, \frac{A_{l+1}}{I^{n+1}A_{l+1}}  \Big)  \rightarrow
      \ext^i_A \Big( k, \frac{A_l(\underline y)}{I^{n+1}A_l(\underline y)} \Big)  \rightarrow 0\]
      It now follows that $ \varepsilon^i_{A_l(\underline y),I}(n+1) = \varepsilon^i_{A_{l+1},I}(n+1) - \varepsilon^i_{A_{l+1},I}(n)$ for all $ n >> 0.$
      Now by induction hypothesis $ \deg  \varepsilon^i_{A_l(\underline y),I}(x) = l-1,$ we have that  $\deg  \varepsilon^i_{A_{l+1},I}(x) = l.$
       Notice that if $l=d$ we have $A_d= A.$ So we obtain in this case that $\deg  \varepsilon^i_{I}(x) = d-1.$

   \end{proof}

\section{\bf Betti Numbers of $A/I^n$}
 Let $(A, \m)$ be a \CM ~ local ring of dimension $d \geq 1$ with infinite residue field $k.$ We
 consider the following numerical function
 \[ n \mapsto \ell(\Tor_i^A(k, A/I^n)) \]
 It is well known that this numerical function is given by a
 polynomial for $ n >> 0$ of degree at most $d-1$ for $ i \geq 1,$  see \cite[Theorem 2]{Kod}. We denote this polynomial by $t^A_{I,i}(k, z).$ Note that   $\ell(\Tor_i^A(k,
 A/I^n))$ gives the $i^{th}$ Betti number of $A/I^n.$ We show that
 if $ \curv(I^n) > 1$ for all $ n \geq 1$ then $\deg t^A_{I,i}(k,
 z)$ is exactly $d-1.$

\begin{lemma}\label{tornonzero}
    Let $(A, \m)$ be a \CM ~ local ring of dimension $d=1.$ Suppose $I$ is  a non-principal $\m$-primary ideal of $A.$
    Then $\Tor_i^A(k, A/I^n)  \neq 0  $ for any $ n \geq 1$ and $ i \geq 1.$

\end{lemma}
\begin{proof} Suppose on the contrary $\Tor_i^A(k, A/I^n)  = 0  $ for some $n \geq 1$ and $ i \geq 1.$  Then $\projdim A/I^n < \infty .$ By the Auslander-Buchbaum formula  $\projdim A/I^n $ is either $0$ or $1.$ If $\projdim A/I^n = 0$ then $A/I^n$ is free $A$-module. So $I^n=0,$ a contradiction.  If  $\projdim A/I^n = 1$ then $I^n $ is free $A$-module and so $I^n$ is principal.  Hence by Lemma \ref{principal},  $I$ is principal ideal, a contradiction.

\end{proof}
Recall that for  an $I$-superficial sequence $ \underline x = x_1,
x_2,...,x_{d-1}$ in ideal $I$ of  $A$ and for $1 \leq l \leq d$ we
defined, $A_l(\underline x) = A/(x_1,...,x_{d-l}).$ Assume that
$\htt(I) \geq d-1.$
\begin{lemma} \label{tornonzero} Let $( A, \m)$ be a \CM ~ local ring with $\dim A = d \geq 2.$ Set $A_1= A_1(\underline x).$
Let $I$ be an ideal with $\htt(I) \geq d-1 $ and  $ \curv(I^n)
> 1$ for some $ n \geq 1.$ Then
     \begin{enumerate}
        \item [\rm (i)] $\projdim_A \frac{A_1}{I^nA_1} = \infty$ and
        \item [\rm (ii)] $ \Tor_i^A(k, \frac{A_1}{ I^nA_1}) \neq 0$ for all $ i \geq 1.$
    \end{enumerate}

\end{lemma}
\begin{proof}
  Proof of (i) is the same as in Lemma \ref{extnonzero}. Proof of (ii) follows from (i).
\end{proof}

Notation : $ t^L_{I,i}(M, n) = \ell(\Tor_i^A(M, L/I^{n+1}L))$
\begin{corollary}\label{tordegzero}  Let $( A, \m)$ be a \CM ~ local ring with $\dim A = d \geq 2.$
Let $I$ be an ideal with $\htt(I) \geq d-1$ and  $ \curv(I^n) > 1$
for all $ n \geq 1.$ Then for all $i \geq 1, $ we have $\deg
t^{A_1}_{I,i}(k,z) =0.$
\end{corollary}
\begin{proof} As $\dim A_1 =1,$ we know from \cite[Theorem 2]{Kod} that $\deg t^{A_1}_{I,i}(k,z) \leq 0.$ Now by Lemma \ref{tornonzero} above we obtain $ \Tor_i^A(k, \frac{A_1}{ I^nA_1}) \neq 0$ for all $n \geq 1.$ This shows that $\deg t^{A_1}_{I,i}(k,z) = 0.$

\end{proof}
\begin{lemma}\label{fingentor} Let $A$ be a local ring of depth $d$ and $M$ be a finitely generated $A$-module. Suppose that $\projdim M < \infty $ and $\depth(M)=k \geq 1.$ Then for all $ i \geq d-k+1 $ we have $ \Tor_i^A(k, L^I(M))$ is finitely generated $\mathcal R(I)$-module
\end{lemma}
\begin{proof}
    Consider the following exact sequence
       \[ 0 \rightarrow \mathcal R(I, M) \rightarrow M[t] \rightarrow L^I(M)(-1)\rightarrow 0. \]
        We set $\mathcal R(M) = \mathcal R(I, M)$ for the rest of the proof. Note that for $ i \geq d-k+1 $ we have $\Tor_i^A\big(k,M[t]\big) =0. $
        Hence applying the functor $ - \otimes k $ we get the following long exact sequence for $ i \geq d-k+1 ,$
         \[ \quad \longrightarrow \Tor_{i+1}^A\big(k, \mathcal R(M)\big) \longrightarrow 0 \longrightarrow \Tor_{i+1}^A\big(k,  L^I(M))\big(-1) \]
          \[ \longrightarrow \Tor_i^A\big(k, \mathcal R(M)\big) \longrightarrow 0 \longrightarrow \Tor_i^A\big(k,  L^I(M))\big(-1) \]

      $ \quad \quad \quad \quad \quad    \longrightarrow \Tor_{i-1}^A\big(k, \mathcal R(M)\big) \longrightarrow \Tor_{i-1}^A\big(k,M[t]\big)  \longrightarrow $\\
      So for $i \geq d-k+1 $ we have,  \[\Tor_{i+1}^A\big(k,  L^I(M)\big)(-1) \cong \Tor_i^A\big(k, \mathcal R(M)\big)\]
      and for $i=d-k+1$ we have,
       \[ 0 \longrightarrow \Tor_{d-k+1}^A\big(k,  L^I(M)\big)(-1) \longrightarrow  \Tor_{d-k}^A\big(k, \mathcal R(M)\big)  \]
       Since $\Tor_i^A\big(k, \mathcal R(M)\big)$ and $ \Tor_{d-k}^A\big(k, \mathcal R(M)\big)$ are finitely generated $\mathcal R(I)$-modules,  it follows that $ \Tor_i^A\big(k, L^I(M)\big)$ is finitely generated $\mathcal R(I)$-module for $ i \geq d-k+1.$

\end{proof}

\begin{theorem}   Let $(A, \m)$ be a \CM ~  ring of dimension $d \geq1.$ Let $I$ be an ideal with $\htt(I) \geq d-1$ and $\curv(I^n) >1 $ for all $ n \geq 1. $ Then for $i \geq d-1,$
    \[ n \mapsto   \ell ( \Tor_i^A(k, A/I^n)) \]
    is given by a polynomial  $t^A_{I,i}(k, z)$ of degree $d-1$ for $ n >> 0.$
\end{theorem}
\begin{proof} We may assume that the residue field is infinite. Also by \cite[Theorem 2]{Kod} we know that $\deg t^A_{I,i}(k, z) \leq d-1. $  First let $d=1.$ As $\curv I^n >1 $ for all $ n \geq 1$ we get $\ell ( \Tor_i^A(k, A/I^n)) \neq 0$ for all $i \geq 1$ and $ n \geq 1.$ So $\deg t^A_{I,i}(k, z) = d-1 = 0.$\\ \indent Now assume $d \geq 2.$  We know that the following function
    \[n \mapsto  \ell ( \Tor_i^A(k, A_l/I^nA_l)) \]
    is given by a polynomial $t^{A_l}_{I,i}(k, n)$ of degree atmost $l-1.$
        Now we observe that $ \depth A_{l+1}=l+1,$ so by the previous Lemma \ref{fingentor} we find that  $ \Tor_s^A(k, L^I(A_{l+1})$ is finitely generated graded $\mathcal R(I)$-module for all $s \geq d-l.$  So for $ i \geq d-l+1$ we can choose $ y \in I$ satisfying the following two properties,
    \begin{enumerate}
        \item [\rm (i)] $yt \in \mathcal R(I)_1$ is $\Tor^A_s(k, L^I(A_{l+1}))$-filter regular for $ s= i, ~ i-1.$
        \item [\rm (ii)] $y$ is $I$-superficial w.r.t. $A_{l+1}.$
    \end{enumerate}
We now extend the superficial sequence $ \underline y' = x_1, \ldots , x_{d-l-1},y$ to a maximal $I$-superficial sequence $\underline y.$ Notice that $A_{l+1}(\underline x)= A_{l+1}(\underline y).$ As in Theorem \ref{theorem1} we have the following short exact sequence,
    \[ 0 \longrightarrow \frac{A_{l+1}}{I^nA_{l+1}}  \xrightarrow[\hspace{17pt}]{\mu_y} \frac{A_{l+1}}{I^{n+1}A_{l+1}} \longrightarrow \frac{A_l(\underline y)}{I^{n+1}A_l(\underline y)}\longrightarrow 0.\]
    Applying the functor $ -\otimes k $ we get for $ i \geq d-l+1,$
    \[ 0 \rightarrow   \Tor_i^A\Big(k,\frac{A_{l+1}}{I^nA_{l+1}} \Big) \longrightarrow \Tor_i^A\Big(k, \frac{A_{l+1}}{I^{n+1}A_{l+1}} \Big) \longrightarrow   \Tor_i^A\Big(k,\frac{A_l(\underline y)}{I^{n+1}A_l(\underline y)}\Big) \rightarrow 0 \]
    So for $n >> 0$ it follows that for all $ i \geq d-l+1,$
    $$t^{A_l(\underline y)}_{I,i}(k, n+1) =t^{A_{l+1}}_{I,i}(k, n+1) - t^{A_{l+1}}_{I,i}(k, n). $$

      We now claim that the $\deg t^{A_l}_{I,i}(k, z)= l-1$ for all  $ i \geq d-1.$ We prove the claim by induction on $l.$ When $l=1$ we have from the Corollary \ref{tordegzero} above that $\deg t^{A_1}_{I,i}(k, z)= 0 $ for $ i \geq 1$ . By above, when $l=2,$ we have for all $ i \geq d-2+1=d-1,$
      $$t^{A_1(\underline y)}_{I,i}(k, n+1) =t^{A_{2}}_{I,i}(k, n+1) - t^{A_{2}}_{I,i}(k, n) $$  Since
      $\deg t^{A_1}_{I,i}(k, z)= 0 $ for $ i \geq 1,$ we get $\deg t^{A_2}_{I,i}(k, z)= 1 $ for all $ i \geq d-1.$ Now assuming the claim is true for $l,$ we prove it for $l+1,$ i.e. we show that $\deg t^{A_{l+1}}_{I,i}(k, z)= l$ for $ i \geq  d-1.$
     By above, we have for $ i \geq d-(l+1) +1=d-l$ that
    \[ t^{A_l(\underline y)}_{I,i}(k, n+1) =t^{A_{l+1}}_{I,i}(k, n+1) - t^{A_{l+1}}_{I,i}(k, n). \]
    Now by induction hypothesis we have $ \deg t^{A_l(\underline y)}_{I,i}(k, z)= l-1 $ for $  i \geq d-1.$ Hence  $\deg t^{A_{l+1}}_{I,i}(k, z)= l$ for $i \geq \max\{d-1, d-l \}=d-1.$ This proves that  $\deg t^{A_l}_{I,i}(k, z)= l-1$ for all  $ i \geq d-1.$  Notice that when $ l=d$ we have $A_d=A.$ So $ \deg t^{A}_{I,i}(k, z)= d-1 $ for all $ i \geq d-1.$
       \end{proof}

\section{\bf Hilbert Polynomials Associated to Derived Functors}

Let $(A, \m)$ be a Noetherian \CM ~ local ring of dimension $d
\geq 1$ with infinite residue field $k$ and $I$ be an $\m$-primary
ideal in $A.$ Let $M$ be a finitely generated maximal \CM ~
$A$-module. We now consider the following numerical function for $
i \geq 1,$
\[ n \mapsto \ell(\Tor_i^A(M, A/I^{n+1}A)) \]
It is known from \cite[Theorem 2]{Kod} that this function coincides with a polynomial denoted by $t_{I, i}^A(M, n)$ for $ n >> 0$ of degree at most $d-1.$ We now recall the notion of reduction of an ideal. We say that $J \subseteq I$ is a reduction of $I$ if
there exists a natural number $m$ such that $JI^n=I^{n+1}$ for all
$ n \geq m.$ We define $r_J(I)$ to be the least such $m.$ A
reduction $J$ of $I$ is called minimal if it is minimal with
respect to inclusion. Reduction number of $I$  is defined as
follows,
$$r(I)= \min \{ r_J(I) \mid  J ~ is ~ minimal ~ reduction ~ of ~ I  \}.$$

\begin{lemma}\label{iso1} Let $A$ be a \CM  $~$ local ring of dimension $d
    \geq 1$ and $I$ be an ideal of $A$ such that $\underline a= a_1,
    a_2, \ldots a_n$ be a regular sequence in $I$ such that $I^2
    =(\underline a) I.$ Then
    \[ \frac{(\underline a)^{k}}{I^{k+1}} \cong  \Big(\frac{A}{I}\Big)^{k+n-1 \choose n-1}. \]
\end{lemma}
\begin{proof}

    Note that by \cite[Theorem 1.1.8]{BH}

    \[\frac{(\underline a)^{k}}{(\underline a)^{k+1}} \cong \Big(\frac{A}{(\underline
        a)}\Big)^t \] where $t={k+n-1 \choose n-1}.$ Tensoring with $A/I$
    gives
    \begin{align*}
    \frac{(\underline a)^{k}}{(\underline a)^{k+1}} \otimes \frac{A}{I}   & \cong \Big(\frac{A}{(\underline a)}\Big)^t\otimes \frac{A}{I} \\
    \frac{(\underline a)^{k}}{I(\underline a)^{k}+(\underline
        a)^{k+1}} & \cong \Big(\frac{A}{(\underline a)}\otimes \frac{A}{I}
    \Big)^t \\
    \frac{(\underline a)^{k}}{I^{k+1}}  & \cong
    \Big(\frac{A}{I}\Big)^t
    \end{align*}
    The last isomorphism holds true because $I(\underline a)^{k}= I^{k+1}$  and $(\underline
    a)^{k+1} \subseteq  I^{k+1}.$

\end{proof}

\begin{proposition}\label{torpoly1} Let $A$ be a \CM  $~$ local ring of dimension $d
    \geq 1$ with infinite residue field and $M$ be a maximal \CM $~$$A$-module. Let $I$ be an
    $\m-$primary ideal of $A$ with  $r(I) \leq 1.$  Then for $i \geq
    1$ we have
    \[\ell( \Tor_i^A(M, A/I^{n+1})) = \ell( \Tor_i^A(M,
    A/I)){n+d-1 \choose d-1}.\]
\end{proposition}

\begin{proof} Since $r(I) \leq 1, $ there exists a minimal
    reduction $ J= (a_1, a_2, \ldots , a_d)$ of $I$ such that $I^2=JI$ where
    $\underline a = a_1, a_2, \ldots , a_d$ is an $A-$regular
    sequence. Since $M$ is maximal \CM $~$ $A-$module it follows that
    $\underline a$ is also an $M-$regular sequence. Now consider the
    following exact sequence
    \[ 0 \rightarrow  \frac{(\underline a)^{n}}{I^{n+1}} \rightarrow  \frac{A}{I^{n+1}}\rightarrow  \frac{A}{(\underline a)^{n}}\rightarrow 0.\]
    As $\underline a$ is $M-$regular sequence we have $\ell(
    \Tor_i^A(M, A/(\underline a)^n))=0 $ for $ i \geq 1.$ So for $ i
    \geq 1,$ the long exact sequence of the functor $M\otimes -$ gives
    \[\Tor_i^A\Big(M, \frac{(\underline a)^{n}}{I^{n+1}} \Big) \cong \Tor_i^A(M,
    A/I^{n+1}) \] By the lemma \ref{iso1} above we have
    \[ \frac{(\underline a)^{n}}{I^{n+1}} \cong  \Big(\frac{A}{I}\Big)^{n+d-1 \choose d-1} \]
    Therefore,
    \[\Tor_i^A\Big(M, \frac{(\underline a)^{n}}{I^{n+1}} \Big)\cong \Big(\Tor_i^A(M,\Big(\frac{A}{I}\Big)\Big)^{n+d-1 \choose d-1}
    \]
    Hence
    \[\ell( \Tor_i^A(M, A/I^{n+1})) = {n+d-1 \choose d-1}\ell( \Tor_i^A(M,
    A/I)).\]

\end{proof}

\begin{corollary}Let $A$ be a \CM  $~$ local ring of dimension $d
    \geq 1$ with infinite residue field and $M$ be a maximal \CM $~$$A$-module. Let $I$ be an
    $\m-$primary ideal of $A$ with  $r(I) \leq 1.$  Then for $i \geq
    1$ we have $\deg t_{I,i}^R(M, z)$ is either $-\infty$ or $d-1.$
\end{corollary}

\begin{corollary}\label{coro2}Let $A$ be a \CM  $~$ local ring of dimension $d
    \geq 1$ and $M$ be a non-free maximal \CM $~A$-module. Let $I$ be
    an integrally closed ideal $\m-$primary ideal of $A$ with  $r(I)
    = 1.$ Then for all $ i \geq 1, $ we have  $\deg t_{I,i}^A(M, z)=d-1$  and the
    leading coefficient is $\ell(\Tor_i^A(M, A/I)).$
\end{corollary}

\begin{proof} By the proposition \ref{torpoly1} above we have
    \[t_{I,i}^A(M, z) = \ell( \Tor_i^A(M,
    A/I)) {z+d-1 \choose d-1} . \] Since $I$ is integrally closed ideal
    we have $\Tor_i^R(M, A/I) \neq 0,$ for  if $\Tor_i^A(M, A/I) =0$
    then by \cite[Corollary 3.3]{HunCorso} $ \projdim(M) < i.$ This is
    not possible as $M$  is non-free maximal \CM $~$ $A-$module. Hence
    the degree of $t_{I,i}^A(M, z)$ is exactly $d-1$ in this case.
\end{proof}

\begin{corollary}
    Let $A$ be a \CM  $~$ local ring of dimension $d \geq 1$ and $M$
    be a maximal \CM $~ A$-module. Let $I$ be an $\m-$primary ideal of
    $A$ with  $r(I) \leq 1.$  Then $ e_1^I(A)
    \mu(M)-e_1^I(M)-e_1^I(\syz_1^I(M)) = \ell(\Tor_1^A(M, A/I)).$
\end{corollary}
\begin{proof} By \cite[Proposition 17]{Pu1} we have
    \[t_{I,1}^A(M,z) = \Big((e_1^I(A)
    \mu(M)-e_1^I(M)-e_1^I(\syz_1^I(M))\Big)\frac{z^{d-1}}{(d-1)!}+
    \text{lower terms in z} \] So by the corollary \ref{coro2} above
    we have \[ e_1^I(A) \mu(M)-e_1^I(M)-e_1^I(\syz_1^I(M)) =
    \ell(\Tor_1^A(M, A/I)).\]

\end{proof}

\begin{proposition}\label{primary}
    Let $A$ be a hypersurface ring of dimension $d = 1.$  Let $M$ be a maximal \CM\;$A-$module. Let $I$ be an $\m-$primary ideal
    which is not a parameter ideal. Then the following conditions are equivalent:
    \begin{enumerate}
        \item [\rm (i)] $\deg t_I^A(M, z) < d-1.$
        \item [\rm (ii)] $M$ is free $A-$ module.
    \end{enumerate}
\end{proposition}

\begin{proof}
    Consider the following exact sequence
    \[0  \longrightarrow I^n \longrightarrow A \longrightarrow A/I^n \longrightarrow 0 \]

    Tensoring with $M$ we get
    \[0  \longrightarrow \Tor_1^A(M, A/I^n) \longrightarrow M \otimes_A I^n \longrightarrow M \longrightarrow  M/ I^nM \longrightarrow 0\]

    Suppose $\deg t_{I, 1}^A(M, z) < d-1=0.$ So $\Tor_1^A(M, A/I^n) =0. $

    Since $\dim M=1$ we obtain from the exact sequence above that $M
    \otimes_A I^n$  is maximal \CM\;$A-$module.

    Since $I^n$ is a module of constant rank $1$ it follows from
    Huneke-Weigand theorem
    that atleast one of  $M$ or $I^n$ is free $A-$module.
    If $I^n$ is free $A-$module (of rank $1$) then $I^n=(a)$ for some $A-$regular
    element of $A.$ By Lemma \ref{principal} we get that $I$ is principal ideal generated by a regular element. This contradicts the fact that $I$ is
    non-parameter ideal. Thus $M$ is a free $A-$module.

\end{proof}


\section{\bf Integral Closure Filtration}
Suppose  $\mathcal I= \{ I_n\}$ is an admissible $I$-filtration
of $\m$-primary ideals in $A.$  Then as in the $I$-adic case,
 the numerical function $n \mapsto  \ell( \Tor_i^A(M, A/ I_n)) $ for any $ i \geq 1$  is given by a polynomial for $n >> 0,$ denoted by
 $t_{\mathcal I,i}^{A}(M, z).$ When $A$ is analytically
unramified and $I$ is an ideal of $A$ then by a theorem of D. Rees
\cite{Rees}, it is known that $\mathcal I = \{ \overline{I^n} \}
$ is an  admissible $I$-filtration.
\begin{proposition} \label{unramified} If $(A, \m)$ is analytically unramified Cohen-Macaulay ring of dimension $1$ and $I$ is an $\m$-primary ideal in $A.$
 Let $M$ be a non-free maximal Cohen-Macaulay module and $\mathcal I=\{\overline{I^n}\}$ be the integral closure filtration then
    $\deg t_{\mathcal I,i}^{A}(M, z) =0$ i.e.  $ t_{\mathcal I,i}^{A}(M, z)$ is a non-zero constant polynomial for $ i \geq 1.$
\end{proposition}
\begin{proof} Since $A$ is analytically unramified  ring $\mathcal I=\{\overline{I^n}\}$ is an $I$-admissible filtration of ideals. For $i \geq 1 $ we have
    $\Tor_i^A(M, A/\overline{I^n}) \neq 0. $ This is because if $\Tor_i^A(M, A/\overline{I^n}) =0$ for some $ i \geq 1$ then by
    \cite[Corollary 3.3]{HunCorso} we get $ \projdim(M) < i.$  Since $M$ is Maximal \CM $~$ so we have $M$ is free $A$-module.
     This is a contradiction and so we have
    $\Tor_i^A(M, A/\overline{I^n}) \neq 0 $ as claimed. So $ t_{\mathcal I,i}^{A}(M, z)$ is a non-zero constant polynomial.
\end{proof}

\begin{lemma}\label{Pu}
 Let $(A, \m) $ be a local ring and $M$ a finite non-free $A-$module. Let $I$ be an ideal in $A.$ Let $\mathcal I= \{ I_n\}$ be an admissible $I$-filtration
 of  ideals in $A.$   Denote by $L$ the first syzygy of $M.$ Let $x \in I$ be $\mathcal I$-superficial non-zero divisor on $A,$ $M,$ and $L.$
  Suppose that $\ell(\Tor^{A}_{1}(M, A/I_{n+1})) < \infty $ for $n >> 0.$
 Set $B=A/xA $, $N=M/xM$, $J_n = I_nB$ and $ \overline {\mathcal I} = \{ J_n \}.$ Then  we have
$$t_{\mathcal I}^{A}(M;n) =  t_{\mathcal I}^A(M; n-1) + t_{\overline{ \mathcal I}}^B(N; n) ~~~~ for ~ all~~~ n \gg 0. $$
$$\deg t_{\overline{ \mathcal I}}^B(N;z) \leq \deg t_{\mathcal I}^A(M;z) - 1. $$
\end{lemma}
\begin{proof}
 Since $x$ is $\mathcal I$-superficial for $A,$ one has following exact sequence for all $n \gg 0$
$$0 \longrightarrow  A/I_n  \xrightarrow[\hspace{13pt}]{i} A/I_{n+1} \longrightarrow  B/J_{n+1} \longrightarrow 0$$
 where the map $i$ is defined by $i_n(a+I_n)=xa+I_{n+1}.$ Applying $M \otimes_A-$ to above exact sequence gives the following exact sequence of $A-$modules

 \[ \Tor_1^A(M, A/I_n) \xrightarrow[\hspace{13pt}]{\Tor_1^A(M, i_n)} \Tor_1^A(M, A/I_{n+1}) \xrightarrow[\hspace{13pt}]{} \Tor_1^A(M, B/J_{n+1})  \xrightarrow[\hspace{13pt}]{}\]
 \[  M/I_nM \xrightarrow[\hspace{13pt}]{M\otimes i_n} M/I_{n+1}M  \xrightarrow[\hspace{13pt}]{} N/J_{n+1}N \xrightarrow[\hspace{13pt}]{} 0.\]
    Since $x$ is $\mathcal I$-superficial on $M$ the map $M \otimes i_n$ is injective for $n \gg 0.$ We claim that
     the map $\Tor_1^A(M, i_n)$ is injective for $n \gg 0.$ For this consider the exact sequence defining $L,$ i.e.
     \[0 \longrightarrow L \longrightarrow F \longrightarrow  M \longrightarrow 0. \]
     Now applying the functors $-\otimes_A A/I_n$ and $-\otimes_A A/I_{n+1} $ one gets the  following commutative diagram

     \[
  \xymatrix
{
 0
 \ar@{->}[r]
  & \Tor_1^A(M,A/I_n)
    \ar@{->}[d]^{\Tor_1^A(M,i_n)}
\ar@{->}[r]^{}
 & L/I_nL  \ar@{->}[d]^{L \otimes i_n}
 \\
 0
 \ar@{->}[r]
  & \Tor_1^A(M,A/I_{n+1})
\ar@{->}[r]
 & L/I_{n+1}L
  }
\]
Notice the following \[ Ker(L \otimes i_n) =
\frac{I_{n+1}L:_L x}{I_nL}. \] As $x$ is $\mathcal I$-superficial on $L$ it
follows that the map $L \otimes i_n$ is injective for $n \gg 0.$
Thus
$ \Tor_1^A(M,i_n)$ is injective for $n \gg 0.$ \\
 So for $n \gg 0$ above long exact sequence becomes

\[0 \longrightarrow  \Tor_1^A(M, A/I_n) \longrightarrow   \Tor_1^A(M, A/I_{n+1}) \longrightarrow   \Tor_1^A(M, B/J_{n+1}) \longrightarrow 0. \]
Now since $x$ is both $A-$regular and $M-$regular we obtain the
following isomorphism, see \cite[18.2 ]{Mat}
$$ \Tor_1^A(M, B/J_{n+1}) \cong  \Tor_1^B(N, B/J_{n+1}).$$
From this isomorphism and the exact sequence above it follows that
for $n \gg 0$
\[\ell_A( \Tor_1^A(M, A/I_{n+1})) = \ell_A( \Tor_1^A(M, A/I_n)) + \ell_A(\Tor_1^B(N, B/J_{n+1})).\]
Thus it follows that
$$t_{\mathcal I}^{A}(M;n) =  t_{\mathcal I}^A(M; n-1) + t_{\overline{ \mathcal I}}^B(N; n) ~~~~ for ~ all~~~ n \gg 0. $$
$$\deg t_{\overline{ \mathcal I}}^B(N;z) \leq \deg t_{\mathcal I}^A(M;z) - 1. $$

\end{proof}

\begin{lemma} \label{choice}If $(A, \m)$ is analytically unramified Cohen-Macaulay local ring of dimension $d$ and infinite residue field $k.$ Let $I$
be an $\m$-primary ideal in $A$ and $M$ a maximal Cohen-Macaulay
$A$-module. Let $\mathcal
    I=\{\overline{I^n}\}$ be the integral closure filtration of ideals in $A.$ Suppose $I=(a_1, a_2, \ldots a_l).$  Let
    $A \longrightarrow \hat{A} \longrightarrow B= \hat{A}[X_1, \ldots, X_l]_{\m\hat{A}[X_1, \ldots, X_l] }$ be extension of rings.
     Let $\mathcal J = \{ \overline{I^n}B
    \}$ be a filtration of ideals in $B$ and $ \xi = a_1
    X_1+a_2X_2+ \ldots + a_lX_l .$
    We set $T=\hat{A}[X_1, \ldots, X_l],$ $\hat M = M\otimes_A \hat{A},$ $M_T = \hat M \otimes_{\hat{A}} T,$ $M_B = \hat M \otimes_{\hat A}B.$
    Then,
    \begin{enumerate}
        \item [\rm (i)] $I^{n+1}M_T :_{M_T} \xi = I^{n}M_T $ and  $\overline{I^{n+1}}M_T :_{M_T} \xi = \overline{I^{n}}M_T $  for $ n >> 0.$
         \item [\rm (ii)] $ (B, \m B)$ is analytically unramified \CM ~ local ring.
        \item [\rm (iii)] $I^{n+1}M_B :_{M_B} \xi = I^{n}M_B $ and $\overline{I^{n+1}}M_B :_{M_B} \xi = \overline{I^{n}}M_B $  for $ n >> 0.$
        \item [\rm (iv)]$M_B$ is maximal \CM ~ $B$-module.
        \item [\rm (v)]$\xi$ is $\mathcal J$-superficial, $B$-regular and
        $M_B$-regular element.
        \item [\rm (vi)] $\xi$ is superficial on $M_B$ w.r.t. filtration of ideals $\{ I^nB\}.$
        \item [\rm (vii)]$\xi$ is superficial on $M_B$ w.r.t. $\mathcal J= \{ \overline{I^n}B\} $
    \end{enumerate}

\end{lemma}
\begin{proof}
  Since $\grade(I, M) >
0,$ we may assume that $a_1$ is $M$-regular element,
$I$-superficial on $M$ and also $\mathcal
I=\{\overline{I^n}\}$-superficial element on $M.$ Moreover we may
choose  generating set $a_1, \ldots, a_l$ such that each $a_i$ is
 regular element and also superficial element on $M$ with respect to
both $I$-adic filtration and the integral closure filtration
$\mathcal I=\{\overline{I^n}\}.$ As $A \rightarrow \hat A$ is a
flat extension we have that $a_1$ is $\hat M$-regular element and
superficial on $\hat M$ with respect to both  $I \hat A$-adic
filtration and the integral closure filtration $\{\overline{I^n}
\hat A\}.$

Proof of (i) is given in \cite[Proposition 2.6]{Cat} for the ring case. We adapt the same proof to the case of modules. First note that $M_T = \hat{M}[X_1, \ldots, X_l].$
Now let $F \in  I^{n+1}M_T :_{M_T} \xi. $
 Consider a monomial order on $M_T$ with $X_1 < X_2 <  \ldots X_l$ and let $m X_1^{\alpha_1}X_2^{\alpha_2} \dots X_l^{\alpha_l}$ be
  the smallest term that appears in $F.$ Since $ \xi F \in  I^{n+1}M_T $ we obtain $m \in (I^{n+1}\hat{M} :a_1).$
  Since $(I^{n+1}\hat{M} :a_1) = I^n\hat M$ for $ n >> 0$ we get $m \in I^n \hat
  M.$ Replacing $F$ by $F - m X_1^{\alpha_1}X_2^{\alpha_2} \dots
  X_l^{\alpha_l}$ and repeating the argument we find that all the
  coefficients of $F$ are in  $I^n \hat M.$ So $F \in I^n M_T .$
  This proves $  I^{n+1}M_T :_{M_T} \xi  \subseteq I^n M_T$ for $ n >>
  0.$ The other inclusion is obvious. As $a_1$ is $\{\overline{I^n} \hat A\}$-superficial on $\hat M,$ the proof of
$\overline{I^{n+1}}M_T :_{M_T} \xi = \overline{I^{n}}M_T $  is
similar.

To show that $B$ is analytically unramified we use Rees' criterion which
states that a local ring $(A, \m)$ is analytically unramified if
and only if there exists an $\m$-primary ideal in $A$ and  $ k
\geq 1$ such that $\overline {I^n} \subseteq I^{n-k},$ see
\cite{Rees}. Since $A$ is analytically unramified there exists $ k
\geq 1$ such that $\overline{I^n} \subseteq I^{n-k}.$ So
$\overline{I^n}B \subseteq I^{n-k}B$ and hence $B$ is analytically
unramified.

 Proof of (iii) follows immediately from (i) by localizing at $\m\hat{A}[X_1, \ldots, X_l].$  For (iv) we note that $A \rightarrow B $ is a flat
extension, so it follows from \cite[1.2.16]{BH} that $M_B$ is a
maximal \CM ~ $B$-module. In (v), proof of $\xi$ is $\mathcal
J$-superficial follows from \cite[Proposition 2.6]{Cat}. As noted
in $2.13(1)$ since $\xi$ is $\mathcal J$-superficial we see that
$\xi$ is a regular element in $B.$  Since $\xi$ is $B$-regular and
$M_B$ is maximal \CM ~ $B$-module we see that $\xi$ is
$M_B$-regular. To prove (vi) we observe that $\xi$ is
$M_B$-regular and so by $2.13(3)$ and (iii) above we see that
$\xi$ is superficial on $M_B$ w.r.t. $\{I^nB \}.$

For (vii) note first that $ \xi$ is $M_B$-regular by (v) above.
Also by (iii) above we have $\overline{I^{n+1}}M_B :_{M_B} \xi =
\overline{I^{n}}M_B $  for $ n >> 0.$ So by 2.13 (3) we get that
$\xi$ is superficial on $M_B$ w.r.t. $\mathcal J= \{
\overline{I^n}B\}. $
\end{proof}

\begin{theorem}\label{mainth} If $(A, \m)$ is analytically unramified Cohen-Macaulay local ring of dimension $d$ and $I$ be an $\m$-primary ideal in $A.$
Let $M$ be a non-free maximal Cohen-Macaulay module and $\mathcal
I=\{\overline{I^n}\}$ be the integral closure filtration of ideals in $A$ then
    $\deg t_{\mathcal I,i}^{A}(M, z) =d-1.$
    \end{theorem}
\begin{proof} We prove the theorem by induction on $d.$ By Proposition \ref{unramified} above, the result is true for $d=1.$
 So we may assume that $d \geq 2.$ Suppose $I = (a_1, a_2, \ldots a_l).$
    Consider the following extension of rings
    \[A \longrightarrow \hat{A} \longrightarrow B= \hat{A}[X_1, \ldots, X_l]_{\m\hat{A}[X_1, \ldots, X_l] }\]
    It is well known that $B$ is faithfully flat $A$-algebra. Let $T= \hat{A}[X_1, \ldots, X_l]$ and let $ \xi= \displaystyle \sum_{i=1}^l a_iX_i \in T.$
     Let $\mathcal J = \{ \overline{I^n}B
    \}$ be a filtration of ideals in $B.$ \\ Set $C= B/\xi B$ (with $\n $ as its maximal ideal), $\hat M = M\otimes_A \hat{A},$ $M_B =  M \otimes_{
    A}B$ and $N =M_B/\xi M_B.$ Let $L= \syz_1^A(M)$ and  $L_B= L \otimes_A B.$
  We make the following observations :
    \begin{enumerate}
        \item [\rm (i)] $\overline{I^nB} = \overline{I^n}B$ for all $ n \geq 1.$
        \item[\rm (ii)] $L_B$ is maximal \CM ~ $B$-module and
        $L_B= \syz_1^B(M_B) $
        \item [\rm (iii)] $\xi$ is superficial on $M_B$ and on $L_B$ w.r.t. filtration of ideals $\{ I^nB\}.$
        \item [\rm (iv)] $\overline{IC} = \overline{I}C.$
        \item [\rm (v)] $\overline{I^nC} = \overline{I^n}C$ for all $ n >> 0.$
        \item [\rm (vi)] $(C, \n)$ is analytically unramified \CM ~ local ring.
    \end{enumerate}
For proof of (i) see \cite[Lemma 8.4.2(11)]{Hun-Swan}.  For (ii)
we note that  $L$ is maximal \CM ~ and that $A \rightarrow B $ is
a flat extension, so from \cite[1.2.16]{BH} it follows that $L_B$
is a maximal \CM ~ $B$-module. Also since $ (A, \m) \rightarrow
(B, \n) $ is a faithfully flat extension of rings with $ \n = \m
B,$ it is easy to see that $L_B= \syz_1^B(M_B).$ We proved (iii)
in part(vi) of Lemma \ref{choice} above.
 For proof of (iv) see
\cite[Corollary 3.4]{Cat}  while (v) follows from \cite[Corollary
3.7]{Cat}. To show $(C, \n)$ is analytically unramified we use
Rees' criterion. So it is enough to show that there exists $ k
\geq 1$ such that  $\overline{I^nC} \subseteq I^{n-k}C$ for all $
n\geq 1.$ Since $A$ is analytically unramified there exists $ k
\geq 1$ such that $\overline{I^n} \subseteq I^{n-k}.$ By (v) above
 we have, $\overline{I^nC} = \overline{I^n}C
\subseteq I^{n-k}C$ for $ n >> 0.$ So by using Rees' criterion  we
see that $C$ is analytically
unramified.  To show that $C$ is a \CM ~  ring we note that $B$ is \CM ~ and $\xi$ is $B$-regular. \\

Since $A \rightarrow \hat A$ and $\hat A \rightarrow B $ are flat
extensions we get the following isomorphisms,
$$\Tor_1^A(M, \frac{A}{\overline{I^n}}) \cong \Tor_1^{\hat A}(\hat
M, \frac{\hat A}{\overline{I^n} \hat A}) \cong \Tor_1^B(M_B,
\frac{ B}{\overline{I^n}B}) \cong \Tor_1^B(M_B, \frac{
B}{\overline{I^nB}}) $$ Note that the last isomorphism follows
from (i) above. Hence we have  $$\ell\Big(\Tor_1^A\big(M,
\frac{A}{\overline{I^n}}\big)\Big) = \ell\Big(\Tor_1^B\big(M_B,
\frac{ B}{\overline{I^nB}}\big)\Big).$$

As observed in Lemma \ref{choice}(ii), $ (B, \m B)$ is
analytically unramified \CM ~ local ring. So by Rees' theorem the
filtration $\mathcal J= \{ \overline{I^nB}\}$ is an $I$-admissible
filtration of ideals in $B.$ By (vi) above we have $(C, \n)$ is
analytically unramified \CM ~ local ring, so $\mathcal{ \overline
J }= \{\overline{I^nC} \}$ is admissible $I$-filtration. Note that
since $\xi$ is $M_B$ regular we see that $N$ is maximal \CM ~
$C$-module. Also by Lemma \ref{choice} we have that $\xi$ is
superficial on both $M_B$ and $L_B$ w.r.t. $\mathcal J= \{
\overline{I^n}B\}.$ So by Lemma \ref{Pu} above we obtain,

$$t_{\mathcal J}^{B}(M_B;n) =  t_{\mathcal J}^B(M_B; n-1) + t_{\overline{ \mathcal I}}^C(N; n) ~~~~ for ~ all~~~ n \gg 0. $$
$$\deg t_{\overline{ \mathcal I}}^C(N;z) \leq \deg t_{\mathcal J}^B(M_B;z) - 1. $$
By induction hypothesis $\deg t_{\overline{ \mathcal I}}^C(N;z)  = d-2$ so we have, $ d-1 \leq \deg t_{\mathcal J}^B(M_B;z) .$
Already we have $\deg t_{\mathcal J}^B(M_B;z) \leq d-1 ,$ hence $ \deg t_{\mathcal J}^B(M_B;z) =d-1.$


\end{proof}




\section{\bf Acknowledgment}
Ganesh Kadu would like to thank DST-SERB for the financial
assistance under ECR/2017/00790.

\bibliographystyle{amsplain}
\bibliography{ref}


\end{document}